\documentclass[a4paper,reqno,11pt]{amsart}
\usepackage{amssymb}
\usepackage{amstext}
\usepackage{amsmath}
\usepackage{amscd}
\usepackage{amsthm}
\usepackage{amsfonts}
\usepackage{enumerate}
\usepackage{graphicx}
\usepackage{latexsym}
\usepackage{mathrsfs}
\input xy
\xyoption{all}
\usepackage{lscape}
\usepackage{pstricks}
\usepackage{comment}

\usepackage{url}

\newtheorem{theorem}{Theorem}[section]
\newtheorem{corollary}[theorem]{Corollary}
\newtheorem{lemma}[theorem]{Lemma}
\newtheorem{definition-theorem}[theorem]{Definition-Theorem}
\newtheorem{definition-proposition}[theorem]{Definition-Proposition}
\newtheorem{proposition}[theorem]{Proposition}

\theoremstyle{definition}
\newtheorem{definition}[theorem]{Definition}
\newtheorem{remark}[theorem]{Remark}
\newtheorem{example}[theorem]{Example}

\numberwithin{equation}{theorem}

\newcommand{\add}{\mathsf{add}\hspace{.01in}}
\newcommand{\ind}{\mathsf{ind}\hspace{.01in}}

\renewcommand{\mod}{\mathsf{mod}\hspace{.01in}}

\newcommand{\proj}{\mathsf{proj}\hspace{.01in}}

\newcommand{\cok}{\operatorname{Cok}\nolimits}

\newcommand{\End}{\operatorname{End}\nolimits}
\newcommand{\Ext}{\operatorname{Ext}\nolimits}
\newcommand{\gl}{\operatorname{gldim}\nolimits}

\newcommand{\Hom}{\operatorname{Hom}\nolimits}

\renewcommand{\ker}{\operatorname{Ker}\nolimits}
\newcommand{\op}{\operatorname{op}\nolimits}

\def\kk{{\mathbf k}}

\begin{document}
\title[Strongly quasi-hereditary algebras and rejective subcategories]{Strongly quasi-hereditary algebras and rejective subcategories}

\author{Mayu Tsukamoto}\address{Department of Mathematics, Graduate School of Science, Osaka City University, 3-3-138 Sugimoto, Sumiyoshi-ku, Osaka 558-8585, Japan}
\curraddr{Graduate school of Sciences and Technology for Innovation, Yamaguchi University, 1677-1 Yoshida, Yamaguchi 753-8511, Japan}
\email{tsukamot@yamaguchi-u.ac.jp}

\subjclass[2010]{16G10, 18A40}

\begin{abstract}
Ringel's right-strongly quasi-hereditary algebras are a distinguished class of quasi-hereditary algebras of Cline--Parshall--Scott.
We give characterizations of these algebras in terms of heredity chains and right rejective subcategories.
We prove that any artin algebra of global dimension at most two is right-strongly quasi-hereditary.
Moreover we show that the Auslander algebra of a representation-finite algebra $A$ is strongly quasi-hereditary if and only if $A$ is a Nakayama algebra.
\end{abstract}
\maketitle

\section{Introduction}
\subsection{Background}
{\it Quasi-hereditary algebras} were introduced by Scott {\cite{S}} to study highest weight categories in the representation theory of semisimple complex Lie algebras and algebraic groups.
Cline, Parshall and Scott proved many important results in \cite{{CPS}, {PS}}. 
Ringel introduced a special class of quasi-hereditary algebras called \emph{right-strongly quasi-hereditary} algebras \cite{R}, motivated by Iyama's finiteness theorem of representation dimensions of artin algebras \cite{{I}, {I2}}.  
One of the advantages of right-strongly quasi-hereditary algebras is that they have better upper bound of global dimension than that of general quasi-hereditary algebras \cite[\S 4]{R}.
By \cite{I}, it follows that any artin algebra $A$ can be written as $eBe$ for some right-strongly quasi-hereditary algebra $B$ and an idempotent $e$ of $B$.
This idea is widely applicable and hence right-strongly quasi-hereditary algebras appear in the representation theory frequently.
Also certain important algebras associated with preprojective algebras and elements in Coxeter groups are known to be right-strongly quasi-hereditary, e.g.\ \cite{{GLS}, {IR}}.
We refer to \cite{{C}, {C2}, {E}} for recent results on right-strongly quasi-hereditary algebras.

In this paper, we discuss categorical aspects of right-strongly quasi-hereditary algebras following an approach in \cite[Section 2]{I2}, which is unpublished.
In particular, we give a characterization of right-strongly quasi-hereditary algebras in terms of the following three notions (Theorem \ref{thm2}).

\begin{itemize}
\item right-strongly heredity chains (Definition \ref{sthered}),

\item total right rejective chains (Definition \ref{rejch}),

\item coreflective chains (Definition \ref{refch}).
\end{itemize}

As application, we sharpen a well-known result of Dlab and Ringel \cite[Theorem 2]{DR} stating that any artin algebra of global dimension at most two is quasi-hereditary.
We prove that such an algebra is always right-strongly (resp.\ left-strongly) quasi-hereditary (Theorem \ref{gl}).
We give a detailed proof following the strategy of \cite[Theorem 3.6]{I2}.
Moreover we show that the Auslander algebra of a representation-finite algebra $A$ is strongly quasi-hereditary if and only if $A$ is a Nakayama algebra.

\subsection{Our results}

Recall that right-strongly (resp.\ left-strongly) quasi-hereditary algebras are defined as quasi-hereditary algebras whose standard modules have projective dimension at most one (Definition \ref{left}).
Our starting point is the following observation which gives a characterization of right-strongly (resp.\ left-strongly) quasi-hereditary algebras in terms of heredity chains.

\begin{proposition} [Proposition \ref{lem1}]
Let $A$ be an artin algebra. Then $A$ is right-strongly $($resp.\ left-strongly$)$ quasi-hereditary if and only if there exists a heredity chain 
\begin{equation*}
A=H_{0}>H_{1}> \cdots > H_i > \cdots >H_{n}=0
\end{equation*}
such that $H_i$ is a projective right $($resp.\ left$)$ $A$-module for any $0 \leq i \leq n-1$.
\end{proposition}
We call such a heredity chain a {\it right-strongly $($resp.\ left-strongly$)$ heredity chain}.

Moreover we give categorical interpretations of right-strongly (resp.\ left-strongly) heredity chains.
For an artin algebra $A$, there exists a bijection between idempotent ideals of $A$ and full subcategories of the category $\proj A$ of finitely generated projective $A$-modules given by $AeA \mapsto \add eA$.
This gives a bijection between chains of idempotent ideals of $A$ and chains of full subcategories of $\proj A$.
A key idea of this paper is to translate properties of idempotent ideals into properties of full subcategories of $\proj A$.

For an artin algebra $A$ and an arbitrary factor algebra $B$ of $A$, we naturally regard $\mod B$ as a full subcategory of $\mod A$.
In this case, each $X \in \mod A$ has a right (resp.\ left) $(\mod B)$-approximation of $X$ which is monic (resp.\ epic) in $\mod A$. 
More generally, subcategories of an additive category with these properties are called {\it right $(${\rm resp}.\ left$)$ rejective} in {\cite{{I3}, {I2}}}. 
They are a special class of {\it coreflective $(${\rm resp}.\ reflective$)$ subcategories} (see Definition \ref{refsub}) appearing in the classical theory of localizations of abelian categories \cite{St}.

Using the notion of right rejective (resp.\ left rejective, coreflective, reflective) subcategories, we introduce the notion of {\it total right rejective $(${\rm resp}.\ total left rejective, coreflective, reflective$)$ chains} of an additive category (Definitions \ref{rejch}, \ref{refch}).
The following main theorem in this paper characterizes right-strongly (resp.\ left-strongly) quasi-hereditary algebras in terms of these chains.

\begin{theorem} [Theorem \ref{thm1} and Theorem \ref{thm2}] \label{thm1.2}
Let $A$ be an artin algebra and 
\begin{equation} \label{srrc}
A=H_{0}>H_{1}> \cdots > H_i > \cdots >H_{n}=0
\end{equation} 
a chain of idempotent ideals of $A$. 
For $0 \leq i \leq n-1$, we write $H_i = Ae_iA$, where $e_i$ is an idempotent of $A$. 
Then the following conditions are equivalent:
\begin{itemize}
\item[{\rm (i)}] \eqref{srrc} is a right-strongly $($resp.\ left-strongly$)$ heredity chain. 

\item[{\rm (ii)}] The following chain is a total right $($resp.\ left$)$ rejective chain of $\proj A$. 
\[
\proj A=\add e_0 A \supset \add e_1 A \supset \cdots \supset \add e_n A =0.
\] 

\item[{\rm (iii)}] \eqref{srrc} is a heredity chain of $A$ and the following chain is a coreflective $($resp.\ reflective$)$ chain of $\proj A$. 
\[
\proj A=\add e_0 A \supset \add e_1 A \supset \cdots \supset \add e_n A =0.
\] 
\end{itemize}
\end{theorem}

We apply total right (resp.\ left) rejective chains to study right-strongly (resp.\ left-strongly) quasi-hereditary algebras.
We give the following result by combining \cite[Theorem 3.6]{I2} and Theorem \ref{thm1.2}.

\begin{theorem} [Theorem \ref{gl}]
Let $A$ be an artin algebra. 
If $\gl A \leq2$, then $A$ is a right-strongly $($resp.\ left-strongly$)$ quasi-hereditary algebra.
\end{theorem} 

An artin algebra which has a heredity chain such that it is a right-strongly heredity chain and a left-strongly heredity chain is called a \emph{strongly quasi-hereditary algebra}.
They have global dimension at most two \cite{R}, but algebras with global dimension at most two are not necessarily strongly quasi-hereditary.
Applying our results on rejective chains, we give the following characterization of Auslander algebras to be strongly quasi-hereditary.

\begin{theorem} [Theorem \ref{Aus}]
Let $A$ be a representation-finite artin algebra and $B$ the Auslander algebra of $A$.
Then $B$ is a strongly quasi-hereditary algebra if and only if $A$ is a Nakayama algebra $($see {\rm {\cite[\S 4.2]{ARS}}} for the definition of Nakayama algebras$)$.
\end{theorem} 

Note that Theorem 1.4 can be deduced from a recent result {\cite[Theorem 3]{E}}, which is shown by a different method.

\section{Preliminaries}

\subsection*{Notation}
For background materials in representation theory of algebras, we refer to \cite{{ASS}, {ARS}}.
Let $A$ be an artin algebra. Let $J(A)$ be the Jacobson radical of $A$. 
We denote by $\gl A$ the global dimension of $A$.
We write $\mod A$ for the category of finitely generated right $A$-modules and $\proj A$ for the full subcategory of $\mod A$ consisting of the finitely generated projective $A$-modules.
For $M \in \mod A$, we denote by $\add M$ the full subcategory of $\mod A$ whose objects are direct summands of finite direct sums of copies of $M$.

We fix a complete set of representatives of isomorphism classes of simple $A$-modules $\{ S(i) \; | \; i \in I \}$.
For $i \in I$, we denote by $P(i)$ the projective cover of $S(i)$. 
For $X \in \mod A$, we write $[X:S(i)]$ for the composition multiplicity of $S(i)$. 
We denote by $\kk$ a field.

\subsection{Quasi-hereditary algebras and highest weight categories}
We start with recalling definitions of quasi-hereditary algebras and highest weight categories.

\begin{definition}[Cline--Parshall--Scott \cite{CPS}, Dlab--Ringel \cite{DR}]
Let $A$ be an artin algebra. 
 
\begin{itemize}
\item[(1)] A two-sided ideal $H$ of $A$ is called \emph{heredity} if it satisfies the following conditions:
\begin{enumerate}
\item[(a)] $H$ is an idempotent ideal (i.e.\ $H^2=H$), or equivalently, there exists an idempotent $e$ such that $H=AeA$ {\cite[Statement 6]{DR}};
\item[(b)] $H$ is projective as a right $A$-module;
\item[(c)] $HJ(A)H=0$.
\end{enumerate}

\item[(2)] A chain of idempotent ideals of $A$
\[
A=H_{0}>H_{1}>\cdots>H_{i}>H_{i+1}>\cdots>H_{n}=0
\]
is called a \emph{heredity chain} if $H_{i}/H_{i+1}$ is a heredity ideal of $A/H_{i+1}$ for $0 \leq i \leq n-1$.

\item[(3)] $A$ is called a \emph{quasi-hereditary algebra} if there exists a heredity chain of $A$.
\end{itemize}
\end{definition}

Quasi-hereditary algebras are strongly related to highest weight categories defined below.
In fact, an artin algebra $A$ is quasi-hereditary if and only if $\mod A$ is a highest weight category {\cite[Theorem 3.6]{CPS}}. 

Let $\leq$ be a partial order on the index set $I$ of simple $A$-modules.
For each $i \in I$, we denote by $\Delta (i)$ the maximal factor module of $P(i)$ whose composition factors have the form $S(j)$, for some $j \leq i$.
The module $\Delta(i)$ is called the {\it standard module} corresponding to $i$.
Let $\Delta:= \{ \Delta(i) \; | \; i \in I \}$ be the set of standard modules.
We denote by $\mathcal{F}(\Delta)$ the full subcategory of $\mod A$ whose objects are the modules which have a $\Delta$-filtration, namely $M \in \mathcal{F}(\Delta)$ if and only if there exists a chain of submodules 
\[
M=M_0 \supseteq M_1 \supseteq \cdots \supseteq M_l=0
\]
such that $M_i/M_{i+1}$ is isomorphic to a module in $\Delta$. 
For $M \in \mathcal{F}(\Delta)$, we denote by $(M: \Delta(i))$ the filtration multiplicity of $\Delta(i)$, which dose not depend on the choice of $\Delta$-filtrations ({cf.\ \cite[A.1 (7)]{D}}).

\begin{definition} [Cline--Parshall--Scott {\cite{CPS}}]
We say that a pair $(\mod A, \leq)$ is a {\it highest weight category} if there exists a short exact sequence 
\begin{equation*}
0 \to K(i) \to P(i) \to \Delta(i) \to 0  
\end{equation*}
for any $i \in I$ with the following properties:
\begin{enumerate}
\item [(a)] $K(i) \in \mathcal{F}(\Delta)$ for any $i \in I$;  
\item [(b)] if $(K(i):\Delta(j)) \not= 0$, then we have $i < j$.
\end{enumerate}
\end{definition}

For a highest weight category $(\mod A, \leq)$ and a refinement $\leq'$ of $\leq$, it is clear that $(\mod A, \leq')$ is also a highest weight category whose standard modules coincide with those of $(\mod A, \leq)$.  
Therefore, without loss of generality, one can assume that the partial order $\leq$ on $I$ is a total order.

To explain a connection between quasi-hereditary algebras and highest weight categories more explicitly, we introduce the following notion.

\begin{definition}[Uematsu--Yamagata {\cite{UY}}]
Let $A$ be an artin algebra.
A chain of idempotent ideals 
\[
A=H_{0}>H_{1}>\cdots>H_{n}=0
\] 
is called {\it maximal} if the length of the chain is the number of simple modules.
\end{definition}

Any heredity chain of an artin algebra can be refined to a maximal heredity chain {\cite[Proposition 1.3]{UY}}.

Let $A$ be an artin algebra with simple $A$-modules $\{ S(i) \; | \; i \in I \}$ and $e_i$ a primitive idempotent of $A$ corresponding to $S(i)$.
Then there is a bijection 
\[
\{ {\rm total \; orders\;  on\; }I \} \overset{1:1} \longleftrightarrow \{ {\rm maximal \; chains \; of \; idempotent \; ideals} \}
\] 
given by setting $H_j:=A(e_{i_{j+1}} + \cdots + e_{i_{n}})A$ and 
\begin{equation} \label{idemp}
(i_1 < i_2 < \cdots < i_j < \cdots < i_{n}) \mapsto (A=H_0> H_1> \cdots > H_j > \cdots > H_{n} ).
\end{equation}

\begin{proposition} [Cline--Parshall--Scott {\cite[\S 3]{CPS}}] \label{cps}
Let $A$ be an artin algebra and $\leq$ a total order on $I$. 

\begin{itemize}
\item[{\rm (1)}] A pair $(\mod A, \leq)$ is a highest weight category with standard modules $\{ \Delta(i_1), \ldots, \Delta(i_n)\}$ if and only if the corresponding maximal chain of idempotent ideals is a heredity chain.

\item[{\rm (2)}] If the condition in $(1)$ is satisfied, then we have $H_{j}/H_{j+1} \cong \Delta(i_j)^{m_j}$ as right $A$-modules for some positive integer $m_j$.
\end{itemize}
\end{proposition}

\subsection{Right-strongly (resp.\ left-strongly) quasi-hereditary algebras}
Now, we recall the following special class of quasi-hereditary algebras.

\begin{definition}[Ringel {\cite[\S 4]{R}}] \label{left}
Let $A$ be an artin algebra and $\leq$ a partial order on $I$.

\begin{itemize}
\item[(1)] We say that a pair $(A, \leq)$ (or simply $A$) is {\it right-strongly quasi-hereditary} if there exists a short exact sequence 
\begin{equation*}
0 \to K(i) \to P(i) \to \Delta(i) \to 0  
\end{equation*}
for any $i \in I$ with the following properties:
\begin{enumerate}
\item [(a)] $K(i) \in \mathcal{F}(\Delta)$ for all $i \in I$;  
\item [(b)] if $(K(i):\Delta(j)) \not= 0$, then we have $i < j$;
\item [(c)] $K(i)$ is a projective right $A$-module, or equivalently the right $A$-module $\Delta(i)$ has projective dimension at most one.
\end{enumerate}

\item[(2)] We say that a pair $(A, \leq)$ (or simply $A$) is {\it left-strongly quasi-hereditary} if $(A^{\op}, \leq)$ is right-strongly quasi-hereditary.
\end{itemize}
\end{definition}

Note that Definition \ref{left} is slightly different from Ringel's original definition (see \cite[\S 4]{R}).
We can easily check that these are equivalent to each other.
Indeed, his definition induces our conditions (a) and (b) because the pair $(\mod A,\leq)$ is a highest weight category by \cite[\S 4 Proposition]{R}.
Moreover, his condition (a) clearly gives our condition (c).
Conversely, his condition (a) follows from our condition (a) and (c). Furthermore, our conditions (a) and (b) induces his condition (b).

As before, for a right-strongly quasi-hereditary algebra $(A, \leq)$ and a refinement $\leq'$ of $\leq$, it is clear that $(A, \leq')$ is also a right-strongly quasi-hereditary algebra whose standard modules coincide with those of $(A, \leq)$.
Therefore, without loss of generality, one can assume that the partial order $\leq$ on $I$ is a total order.

In this paper, for a quiver $Q$ and arrows $\alpha: x \to y$ and $\beta: y \to z$ in $Q$, we denote by $\alpha \beta$ the composition.

\begin{example}
We assume that a natural number $n$ is at least two. Let $A_n$ be the $\kk$-algebra defined by the quiver
\[
\xymatrix@C=25pt@R=15pt{1\ar@<2pt>[r]^{\alpha_1} & 2\ar@<2pt>[l]^{\beta_1} \ar@<2pt>[r]^{\alpha_2} & \cdots \ar@<2pt>[l]^{\beta_2} \ar@<2pt>[r]^{\alpha_{i-1}} & i \ar@<2pt>[l]^{\beta_{i-1}} \ar@<2pt>[r]^{\alpha_{i}} & i+1  \ar@<2pt>[l]^{\beta_{i}} \ar@<2pt>[r]^{\alpha_{i+1}} & \cdots \ar@<2pt>[l]^{\beta_{i+1}} \ar@<2pt>[r]^{\alpha_{n-1}}& n \ar@<2pt>[l]^{\beta_{n-1}}
 }
\]
with relations $\alpha_{i-1} \alpha_i$, $\beta_i \beta_{i-1}$, $\beta_{i-1} \alpha_{i-1} - \alpha_i\beta_i$ for $2 \leq i \leq n-1$ and $\beta_{n-1} \alpha_{n-1}.$ 
The algebra $A_n$ is Morita equivalent to a block of a Schur algebra (see \cite{{DoRe}, {Er}}).

If $n=2$, then the indecomposable projective modules $P(i)$ have the following shape: 
\[
\begin{xy} 
(0,0)*{1}="1", +/d0.4cm/*{2} ,+/d0.4cm/*{1}, "1"+/r1cm/*{2}, +/d0.4cm/="1"*{1}, 
\end{xy}
\]
For the total order $\{1<2\}$, we have $\Delta(1)=S(1)$ and $\Delta(2)=P(2)$, and hence $A_2$ is right-strongly quasi-hereditary.

If $n>2$, then the indecomposable projective modules $P(i)$ have the following shape: 
\[
\begin{xy}
(0,0)*{1}="1", (0,-4)*{2} ,(0,-8)*{1}, "1"+(15.5,0)*{2}="2", "2"+(-5.5,-4)*{1}, "2"+(5.5, -4)*{3}="3", "2"+(0,-8)*{2}, "3"+(10,0)*{\cdots}="c"
,"c"+(15.5,4)*{i}="i", "i"+(-5.5,-4)*{i-1}, "i"+(5.5,-4)*{i+1}="j", "i"+(0,-8)*{i}, "j"+(10,0)*{\cdots}="d",  "d"+(10,4)*{n}="l", "l"+(0,-4)*{n-1}, 
\end{xy}
\]
Thus $A_n$ is quasi-hereditary with respect to the total order $\{1<2< \cdots < n \}$.
However $A_n$ is not right-strongly quasi-hereditary with respect to any order. 

\end{example}

It is well-known that a pair $(A, \leq)$ is quasi-hereditary if and only if so is $(A^{\op}, \leq)$ (see {\cite[Lemma 3.4]{CPS}} and {\cite[Statement 9]{DR}}).
However even if $(A, \leq)$ is right-strongly quasi-hereditary, it dose not necessarily hold that $(A, \leq)$ is left-strongly quasi-hereditary.
Moreover there is an example of a left-strongly quasi-hereditary algebra $A$ which is not right-strongly quasi-hereditary for any order on $I$ (see \cite[A2 (1)]{R}). 
  
\begin{example} \label{eg}
Let $B$ be the $\kk$-algebra defined by the quiver
\[
\xymatrix@=15pt{ & 1 \ar[rd]^{\gamma} & \\
\ar[ru]^{\alpha} 2 && 3 \ar[ll]^{\beta}
 }
 \]
with relations $\alpha \gamma$, $\beta \alpha$. 
Then the indecomposable projective $B$-modules $P(i)$ have the following shape: 
\[
\begin{xy}
(0,0)*{1}="1", (0,-4)*{3} ,(0,-8)*{2}, "1"+(10,0)*{2}="2", "2"+(0,-4)*{1}, "2"+(10,0)*{3}="3", "3"+(0,-4)*{2} 
\end{xy}
\]
For the total order $\{1<2<3\}$, we have $\Delta(1)=S(1)$ and $\Delta(i)=P(i)$ for $i=2, 3$, and hence $B$ is right-strongly quasi-hereditary. 
On the other hand, the indecomposable projective $B^{\op}$-modules have the following shape:
\[
\begin{xy}
(0,0)*{1^{\op}}="1", (0,-4)*{2^{\op}}, "1"+(10,0)*{2^{\op}}="2", "2"+(0,-4)*{3^{\op}}, "2"+(0,-8)*{1^{\op}}, "2"+(10,0)*{3^{\op}}="3", "3"+(0,-4)*{1^{\op}} 
\end{xy}
\]
For the total order $\{1<2<3\}$, we have $\Delta^{\op}(i)=S^{\op}(i)$ for $i=1, 2$ and $\Delta^{\op}(3)=P^{\op}(3)$, and hence $B$ is not left-strongly quasi-hereditary.
However, for the total order $\{ 2<1<3\}$, $B$ is not right-strongly quasi-hereditary but $B$ is left-strongly quasi-hereditary.
\end{example}

\section{Characterizations of right-strongly quasi-hereditary algebras}

\subsection{Right-strongly heredity chains}

In this subsection, we give a characterization of right-strongly (resp.\ left-strongly) quasi-hereditary algebras in terms of heredity chains.

\begin{definition} \label{sthered}
Let $A$ be an artin algebra and 
\begin{equation} \label{hc}
A=H_{0}>H_{1}> \cdots > H_i > H_{i+1}> \cdots >H_{n}=0
\end{equation}
a chain of idempotent ideals.

\begin{itemize}
\item[(1)] We call \eqref{hc} a {\it right-strongly $($resp.\ left-strongly$)$ heredity chain} if the following conditions hold for any $0 \leq i \leq n-1$:
\begin{enumerate}
\item[(a)] $H_i$ is projective as a right (resp.\ left) $A$-module;
\item[(b)] $(H_i/H_{i+1}) J(A/H_{i+1}) (H_i/H_{i+1}) =0$.
\end{enumerate}

\item[(2)] We call \eqref{hc} a {\it strongly heredity chain} if the following conditions hold for any $0 \leq i \leq n-1$:
\begin{enumerate}
\item[(a)] $H_i$ is projective as a right $A$-module and as a left $A$-module;
\item[(b)] $(H_i/H_{i+1}) J(A/H_{i+1}) (H_i/H_{i+1}) =0$.
\end{enumerate}
\end{itemize}
\end{definition} 
 
\begin{proposition} \label{rshchc}
Any right-strongly $($resp.\ left-strongly$)$ heredity chain of $A$ is a heredity chain.
\end{proposition}

\begin{proof}
Let \eqref{hc} be a right-strongly heredity chain.
It is enough to show that $H_i/H_{i+1}$ is projective as a right $(A/H_{i+1})$-module for any $0 \leq i \leq n-1$.
Since \eqref{hc} is a right-strongly heredity chain, we have that $H_i$ is projective as a right $A$-module for any $0 \leq i \leq n-1$.
Hence $H_i \otimes_A (A/H_{i+1})=H_i/H_{i+1}$ is projective as a right $(A/H_{i+1})$-module for any $0 \leq i \leq n-1$.
\end{proof}
 
\begin{example} 
Let $A$ be an artin algebra.
Then $A$ is hereditary if and only if any chain of idempotent ideals of $A$ is a strongly heredity chain. 
\end{example}

\begin{proof}
The ``only if'' part is clear.
By {\cite[Theorem 1]{DR}}, $A$ is hereditary if and only if any chain of idempotent ideals of $A$ is a heredity chain.
Therefore the ``if'' part follows.
\end{proof}

\begin{example}
Any heredity chain of length at most two is clearly a right-strongly (resp.\ left-strongly) heredity chain. 
\end{example}

\begin{example}
Let $A$ be the Auslander algebra of the truncated polynomial algebra $\kk[x]/(x^n)$.
Then $A$ is given by the quiver 
\[
\xymatrix@C=25pt@R=15pt{1\ar@<2pt>[r]^{\alpha_1} & 2\ar@<2pt>[l]^{\beta_1} \ar@<2pt>[r]^{\alpha_2} & \cdots \ar@<2pt>[l]^{\beta_2} \ar@<2pt>[r]^{\alpha_{i-1}} & i \ar@<2pt>[l]^{\beta_{i-1}} \ar@<2pt>[r]^{\alpha_{i}} & i+1  \ar@<2pt>[l]^{\beta_{i}} \ar@<2pt>[r]^{\alpha_{i+1}} & \cdots \ar@<2pt>[l]^{\beta_{i+1}} \ar@<2pt>[r]^{\alpha_{n-1}}& n \ar@<2pt>[l]^{\beta_{n-1}}
 }
\]
with relations $\beta_i \alpha_i-  \alpha_{i+1} \beta_{i+1}$ ($1 \leq i \leq n-2$), $\beta_{n-1} \alpha_{n-1}$. Then 
\begin{equation*}
A > A(e_2 + \cdots +e_n)A > \cdots  >Ae_nA > 0
\end{equation*}
is a strongly heredity chain of $A$.
This example can be explained by Theorem \ref{Aus}.
\end{example}

We prepare the following easy observation.

\begin{lemma} \label{lem0}
Let $A$ be an artin algebra and 
\begin{equation*}
A=H_0> H_1 > \cdots > H_i > \cdots > H_{n}=0
\end{equation*}
a chain of two-sided ideals.
Then the following conditions are equivalent:

\begin{itemize}
\item[{\rm (i)}] $H_i$ is projective as a right $($resp.\ left$)$ $A$-module for $0 \leq i \leq n-1$.

\item[{\rm (ii)}] The projective dimension of $H_i/H_{i+1}$ as a right $($resp.\ left$)$ $A$-module is at most one for $0 \leq i \leq n-1$.
\end{itemize}
\end{lemma} 
 
\begin{proof}
(i) $\Rightarrow$ (ii): This is clear from the short exact sequence $0 \to H_{i+1} \to H_i \to H_i/H_{i+1} \to 0$.

(ii) $\Rightarrow$ (i): 
Since $0 \to H_{1} \to H_{0} \to H_{0}/H_{1} \to 0$ is a short exact sequence such that $H_{0}=A$ is a projective $A$-module, $H_{1}$ is also projective as a right $A$-module.
Thus we obtain the assertion inductively.  
\end{proof}
 
Now, we are ready to prove the following main observation in this subsection.

\begin{proposition} \label{lem1}
Let $A$ be an artin algebra, $\leq$ a total order on $I$ and 
\begin{equation} \label{rsh}
 A=H_0 > H_1 > \cdots > H_{n}=0
\end{equation}
a maximal chain of idempotent ideals corresponding to $\leq$ by \eqref{idemp}. 
Then $(A, \leq)$ is a right-strongly $($resp.\ left-strongly$)$ quasi-hereditary algebra if and only if \eqref{rsh} is a right-strongly $($resp.\ left-strongly$)$ heredity chain.
\end{proposition}

\begin{proof}  
Both conditions imply that \eqref{rsh} is a heredity chain by Proposition \ref{cps} (1) and Proposition \ref{rshchc}. 
Moreover we have an isomorphism  
\begin{equation} \label{std}
H_j/H_{j+1} \cong \Delta(i_j)^{m_j}
\end{equation}
as right $A$-modules for some positive integer $m_j$ by Proposition \ref{cps} (2).

By \eqref{std}, the pair $(A, \leq)$ is right-strongly quasi-hereditary if and only if the projective dimension of $H_j/H_{j+1}$ as a right $A$-module is at most one for any $0 \leq j \leq n-1$.
By Lemma \ref{lem0}, this is equivalent to that $H_j$ is projective as a right $A$-module for any $0 \leq j \leq n-1$.
Hence \eqref{rsh} is a right-strongly heredity chain.
\end{proof}

Throughout this paper, we frequently use the following basic observations.
 
\begin{lemma} \label{idempi}
Let $A$ be an artin algebra and $e$ an idempotent of $A$.
Then we have the following statements.

\begin{itemize}
\item[{\rm (1)}] If $AeA$ is projective as a right $A$-module, then $AeA \in \add eA$.

\item[{\rm (2)}] If $Ae$ is projective as a right $(eAe)$-module, then the functor $\Hom_A(eA, -) : \mod A \to \mod eAe$ preserves projective modules.
In particular, $\gl eAe \leq \gl A$.

\item[{\rm (3)}] If $AeA$ is projective as a right $A$-module, then $Ae$ is a projective right $(eAe)$-module.
\end{itemize}
\end{lemma}
 
 \begin{proof}
 (1) Take an epimorphism $f : (eAe)^l \twoheadrightarrow Ae$ in $\mod (eAe)$.
 Composing $f \otimes_{eAe} eA : (eA)^l \twoheadrightarrow Ae\otimes_{eAe} eA$ with the multiplication map $Ae \otimes_{eAe} eA \twoheadrightarrow AeA$, we have an epimorphism $(eA)^l \twoheadrightarrow AeA$ of right $A$-modules.
 
 (2) For any $P \in \proj A$, we have that $\Hom_A(eA, P) =Pe$ is a direct summand of a finite direct sum of copies of $\Hom_A(eA,A)=Ae$.
 Hence the assertion holds.
 
 (3) Since $AeA$ is a projective $A$-module, it follows from (1) that $A e A\in \add eA$.
Hence we obtain that $A e = A e A e=\Hom_A(e A, A e A)$ is projective as a right $(e A e)$-module.
 \end{proof}

We end this subsection with the following observations which show that right-strongly (resp.\ left-strongly) quasi-hereditary algebras are closed under idempotent reductions.

\begin{proposition} \label{lem2}
Let $A$ be an artin algebra with a right-strongly $($resp.\ left-strongly$)$ heredity chain  
\begin{equation*}
A=H_{0}>H_{1}> \cdots > H_i > \cdots >H_{n}=0.
\end{equation*}
Then the following statements hold.

\begin{itemize}
\item[{\rm (1)}]
For $0< i \leq n-1$, $A/H_i$ has a right-strongly $($resp.\ left-strongly$)$ heredity chain 
\begin{equation*}
A/H_i=H_{0}/H_i>H_{1}/H_i> \cdots > H_{i}/H_i=0.
\end{equation*}

\item[{\rm (2)}]  
Let $e_i \in A$ be an idempotent of $A$ such that $H_i = Ae_iA$ for $0 \leq i \leq n-1$.
Then $e_i A e_i$ has a right-strongly $($resp.\ left-strongly$)$ heredity chain 
\begin{equation*}
e_i Ae_i=e_i H_i e_i> e_i H_{i+1}e_i> \cdots >e_iH_{n}e_i=0.
\end{equation*}
\end{itemize} 
\end{proposition}

\begin{proof}
(1) It is enough to show that $H_j/H_i$ is projective as a right $(A/H_i)$-module for $1 \leq j <i$.
This is immediate since $H_j$ is projective as a right $A$-module and the functor $- \otimes_A (A/H_i)  : \mod A \to \mod A/H_i$ reflects projectivity.

(2) We prove that $e_iH_j e_i$ is a projective right $(e_iAe_i)$-module.
By Lemma \ref{idempi} (3), we have that $Ae_i$ is projective as a right $(e_iAe_i)$-module.
It follows from Lemma \ref{idempi} (2) that $H_je_i$ is projective as a right $(e_iAe_i)$-module.
Since $H_j e_i = e_i H_j e_i \oplus (1-e_i) H_j e_i$, we have $e_iH_je_i \in \proj (e_iAe_i)$. 
\end{proof}

\subsection{Right rejective subcategories}
In this subsection, we recall the definitions of right rejective subcategories. 
Using them, we characterize right-strongly (resp.\ left-strongly) quasi-hereditary algebras.
We refer to {\cite[Appendix]{ASS}} for background on category theory.

Let $\mathcal{C}$ be an additive category, and put $\mathcal{C}(X, Y):=\Hom_{\mathcal{C}}(X,Y)$. 
In the rest of this paper, {\it we assume that any subcategory is full and closed under isomorphisms, direct sums and direct summands.}
We denote by $\mathcal{J}_{\mathcal{C}}$ the Jacobson radical of $\mathcal{C}$, and by $\ind \mathcal{C}$ the set of isoclasses of indecomposable objects in $\mathcal{C}$.  
For a subcategory $\mathcal{C}'$ of $\mathcal{C}$, we denote by $[\mathcal{C}']$ the ideal of $\mathcal{C}$ consisting of morphisms which factor through some object of $\mathcal{C}'$.  
For an ideal $\mathcal{I}$ of $\mathcal{C}$, the factor category $\mathcal{C}/\mathcal{I}$ is defined by $\mathit{ob}(\mathcal{C}/\mathcal{I}):=\mathit{ob}(\mathcal{C})$ and $(\mathcal{C}/\mathcal{I})(X,Y):= \mathcal{C}(X,Y)/\mathcal{I}(X,Y)$ for any $X, Y \in \mathcal{C}$.  
Recall that an additive category $\mathcal{C}$ is called {\it Krull--Schmidt} if any object of $\mathcal{C}$ is isomorphic to finite direct sum of objects whose endomorphism rings are local.

\begin{definition} [Auslander--Smal$\o$ {\cite{AS}}]
Let $\mathcal{C}$ be an additive category and $\mathcal{C}'$ a subcategory of $\mathcal{C}$. 
We say that $f \in \mathcal{C}(Y,X)$ is a \emph{right $\mathcal{C}'$-approximation} of $X$ if the following equivalent conditions are satisfied.

\begin{itemize}
\item[(i)] $Y \in \mathcal{C}'$ and $\mathcal{C}(- , Y) \xrightarrow{f \circ -} \mathcal{C}(- , X) \to 0$ is exact on $\mathcal{C}'$. 

\item[(ii)] $Y \in \mathcal{C}'$ and the induced morphism $\mathcal{C}(-, Y) \xrightarrow{f \circ-} [\mathcal{C}'](-, X)$ is an epimorphism on $\mathcal{C}$.

Dually, a \emph{left $\mathcal{C}^{'}$-approximation} is defined.  
\end{itemize}
\end{definition}

Now, we introduce the following key notions in this paper.

\begin{definition} [Iyama {\cite[2.1(1)]{I}}] \label{rejsub}
Let $\mathcal{C}$ be an additive category and $\mathcal{C}'$ a subcategory of $\mathcal{C}$.

\begin{itemize}
\item[(1)] We call $\mathcal{C}'$ a \emph{right $(${\rm resp}.\ left$)$ rejective subcategory} of $\mathcal{C}$
if the inclusion functor $\mathcal{C}'\hookrightarrow\mathcal{C}$ has a right (resp.\ left) adjoint with a counit $\varepsilon^-$ (resp.\ unit $\varepsilon^+$) such that $\varepsilon_{X}^-$ is a monomorphism (resp.\ $\varepsilon_{X}^+$ is an epimorphism) for $X\in\mathcal{C}$.

\item[(2)] We call $\mathcal{C}'$ a \emph{rejective subcategory} of $\mathcal{C}$ if $\mathcal{C}'$ is a right and left rejective subcategory of $\mathcal{C}$.
\end{itemize}
\end{definition}

We often use the following equivalent conditions.

\begin{proposition}[Iyama {\cite[Definition 1.5]{I2}}] \label{app}
Let $\mathcal{C}$ be an additive category and $\mathcal{C}'$ a subcategory of $\mathcal{C}$.
Then the following are equivalent:

\begin{itemize}
\item[{\rm (i)}] $\mathcal{C}'$ is a right $($resp.\ left$)$ rejective subcategory of $\mathcal{C}$.

\item[{\rm (ii)}] For any $X\in\mathcal{C}$, there exists a monic right $($resp.\ epic left$)$ $\mathcal{C}'$-approximation $f_X \in \mathcal{C}\left(Y,X\right)$ $($resp.\ $f^X \in \mathcal{C}\left(X,Y\right))$ of $X$.
\end{itemize}
\end{proposition}

\begin{proof}
(i) $\Rightarrow$ (ii): If the inclusion functor $F : \mathcal{C}'\hookrightarrow\mathcal{C}$ has a right adjoint $G$ with a counit $\varepsilon^-$,
then $\varepsilon_{X}^-: G(X) \to X$ is a right $\mathcal{C}'$-approximation of $X \in \mathcal{C}$.
Thus the assertion follows.

(ii) $\Rightarrow$ (i): We assume that, for any $X \in \mathcal{C}$, there exists a monic right $\mathcal{C}'$-approximation of $X$.
We construct a right adjoint functor $G: \mathcal{C} \to \mathcal{C}'$ as follows.
For $X \in \mathcal{C}$, take a monic right $\mathcal{C}'$-approximation $f_X : C_X \to X$.
For a morphism $\varphi \in \mathcal{C}(X, Y)$, there exists a unique morphism $C_{\varphi} : C_X \to C_Y$ making the following diagram commutative.
\[
\xymatrix@C=25pt@R=15pt{ C_X \ar[d]_{C_{\varphi }} \ar[r]^-{f_X} & X \ar[d]^{\varphi} \\
C_Y \ar[r]^-{f_Y} & Y.
 }
\]
It is easy to check that $G(X):=C_X$ and $G(\varphi):=C_{\varphi}$ give a right adjoint functor $G: \mathcal{C} \to \mathcal{C}'$ of the inclusion functor $F: \mathcal{C}'\to \mathcal{C}$ and $f$ gives a counit.
\end{proof}

Right rejective subcategories of $\mod A$ are characterized as follows.

\begin{proposition} [Iyama {\cite[Proposition 1.5.2]{I2}}] \label{cl}
Let $A$ be an artin algebra and $\mathcal{C}$ a subcategory of $\mod A$.
Then $\mathcal{C}$ is a right $($resp.\ left$)$ rejective subcategory of $\mod A$ if and only if $\mathcal{C}$ is closed under factor modules $($resp.\ submodules$)$.
\end{proposition}

\begin{proof}
We show the ``if'' part.
For $M \in \mod A$, we put $G(M):=\sum_{X \in \mathcal{C}, f \in \Hom_A(X, M)} f(X)$.
Then $G(M)$ is a factor module of some module in $\mathcal{C}$.
Thus we have $G(M) \in \mathcal{C}$.
Since the natural inclusion $G(M) \hookrightarrow M$ is a monic right $\mathcal{C}$-approximation of $M$, the assertion holds.

We show the ``only if'' part.
For a surjection $f: M \to N$ with $M \in \mathcal{C}$, we show that $N$ belongs to $\mathcal{C}$.
Since $\mathcal{C}$ is a right rejective subcategory of $\mod A$, there exists a monic right $\mathcal{C}$-approximation $f_N : G(N) \to N$ of $N$.
Thus we have a morphism $g : M \to G(N)$ such that $f=f_N \circ g$.
Since $f$ is surjective, we have that $f_N$ is a bijection.
Hence we have $N \in \mathcal{C}$.
\end{proof}

\begin{proposition} [Iyama {\cite[Theorem 1.6.1(1)]{I2}}] \label{bij}
Let $A$ be an artin algebra.
Then there exists a bijection between factor algebras $B$ of $A$ and rejective subcategories $\mathcal{C}$ of $\mod A$ given by $B \mapsto \mod B$. 
\end{proposition}

\begin{proof}
This is clearly from Proposition \ref{cl} since a full subcategory of $\mod A$ which is closed under submodules and factor modules is precisely $\mod B$ for a factor algebra $B$ of $A$.
\end{proof}

\begin{example} \label{tor}
Let $A$ be an artin algebra.

\begin{itemize}
\item[(a)] Let $(\mathcal{T}, \mathcal{F})$ be a torsion pair on $\mod A$.
Then $\mathcal{T}$ is a right rejective subcategory and $\mathcal{F}$ is a left rejective subcategory of $\mod A$ by Proposition \ref{cl}.

\item[(b)] For a classical tilting $A$-module $T$, we put $\mathcal{T}:=\{ Y\in \mod A \; | \; \Ext_A^1(T,Y)=0 \}$ and $\mathcal{F}:=\{ Y\in \mod A \; | \; \Hom_A(T, Y) =0 \}$.
Then $(\mathcal{T}, \mathcal{F})$ is a torsion pair on $\mod A$, and therefore $\mathcal{T}$ (resp.\ $\mathcal{F}$) is a right (resp.\ left) rejective subcategory of $\mod A$. 

\item[(c)] Assume that $A$ is right-strongly quasi-hereditary and let $T$ be a characteristic tilting module.
Then $T$ is a classical tilting module \cite[Lemma 4.1]{DR6} and hence $(\mathcal{T}, \mathcal{F})$ is a torsion pair on $\mod A$.
Since $\mathcal{T}$ coincides with the subcategory $\mathcal{F}(\Delta)^{\perp}:= \{ Y \in \mod A \;|\; \Ext_A^i (\mathcal{F}(\Delta), Y) =0 \mathrm{\; for \; all\; } i \geq 1 \}$ \cite[\S 4]{DR6}, we have that $\mathcal{F}(\Delta)^{\perp}$ is a right rejective subcategory of $\mod A$.
\end{itemize}
\end{example}

Right rejective subcategories of $\proj A$ are characterized as follows.

\begin{proposition} [Iyama {\cite[Theorem 3.2 (2)]{I2}}] \label{iy1}
Let $A$ be an artin algebra and $e$ an idempotent of $A$. 
Then $\add e A$ is a right $($resp.\ left$)$ rejective subcategory of $\proj A$ if and only if $A e A$ is a projective right $($resp.\ left$)$ $A$-module.
In this case, we have $\gl eAe \leq \gl A$.
\end{proposition}

\begin{proof}
Assume that $\add e A$ is a right rejective subcategory of $\proj A$.
Then there exists  $a \in \Hom_{A}(P, A)$ with $P \in \add (e A)_A$ such that 
\[
P=\Hom_{A}(A, P) \xrightarrow{a \circ -} [\add e A] (A, A) = A e A
\]
is an isomorphism.
Hence $A e A \cong P$ is a projective right $A$-module.

Conversely, we assume that $A e A$ is a projective right $A$-module.
By Lemma \ref{idempi} (1), we have $A e A \in \add e A$ as a right $A$-module.
The inclusion map $i : A e A \hookrightarrow A$ gives a right $(\add e A)$-approximation of $A$ since 
\[
A e = A e A e = \Hom_{A}(e A, A e A) \xrightarrow{i \circ-} \Hom_{A}(e A, A)=A e
\]
is an isomorphism.

In this case, $A e$ is projective as a right $(eAe)$-module by Lemma \ref{idempi} (3).
Thus it follows from Lemma \ref{idempi} (2) that $\gl eAe \leq \gl A$. 

By the duality $\Hom_A(-,A): \proj A \to \proj A^{\op}$, we have that $\add eA$ is left rejective in $\proj A$ if and only if $\add Ae$ is right rejective in $\proj A^{\op}$.
Hence the statement for left rejective subcategories follows.
\end{proof}

To introduce rejective chains, we need the following notion.

\begin{definition}
Let $\mathcal{C}$ be a Krull--Schmidt category.

\begin{itemize}
\item[{\rm (1)}] We call $\mathcal{C}$ a {\it semisimple} category if $\mathcal{J}_{\mathcal{C}}=0$.

\item[{\rm (2)}] A subcategory $\mathcal{C}'$ of $\mathcal{C}$ is called {\it cosemisimple} in $\mathcal{C}$ if the factor category $\mathcal{C}/\left[\mathcal{C}'\right]$ is semisimple. 
\end{itemize}
\end{definition}

We often use the fact that $\mathcal{C}'$ is a cosemisimple subcategory of $\mathcal{C}$ if and only if $[\mathcal{C}'](-, X)=\mathcal{J}_{\mathcal{C}}(-, X)$ holds for any $X \in \ind \mathcal{C} \setminus \ind \mathcal{C}'$.

\begin{lemma} \label{coss}
Let $A$ be an artin algebra and $Ae'A \subset AeA$ idempotent ideals of $A$.
Then the following conditions are equivalent:

\begin{itemize}
\item[{\rm (i)}] $\add e'A$ is a cosemisimple subcategory of $\add eA$.

\item[{\rm (ii)}] $J(eAe/eAe'Ae)=0$.

\item[{\rm (iii)}] $(AeA/Ae'A) J(A/Ae'A) (AeA/Ae'A) =0$.
\end{itemize}
\end{lemma}

\begin{proof}
(i) $\Leftrightarrow$ (ii): Let $\mathcal{C}:=\add eA /[\add e'A]$.
The condition (i) means $\mathcal{J}_{\mathcal{C}}=0$.
This is equivalent to (ii) since $\mathcal{J}_\mathcal{C}(eA, eA) = J(\End_\mathcal{C}(eA))=J(eAe/eAe'Ae)$. 

(ii) $\Leftrightarrow$ (iii): Since $J(eAe/eAe'Ae)=eJ(A/Ae'A)e$, we have the assertion.
\end{proof}

Now, we introduce the following central notion in this paper.

\begin{definition}[Iyama {\cite[2.1(2)]{I}}, {\cite[Definition 2.2]{I2}}] \label{rejch}
Let $\mathcal{C}$ be an additive category and 
\begin{equation} \label{irc}
\mathcal{C}_0 \supset \mathcal{C}_1 \supset \cdots \supset \mathcal{C}_n =0
\end{equation}
a chain of subcategories. 

\begin{itemize}
\item[(1)] We call \eqref{irc} a \emph{rejective chain $(${\rm resp}.\ right rejective, left rejective$)$} if $\mathcal{C}_i$ is a cosemisimple rejective $($resp.\ right rejective, left rejective$)$ subcategory of $\mathcal{C}_{i-1}$ for $1 \leq i \leq n$.

\item[(2)] We call \eqref{irc} a \emph{total right $(${\rm resp}.\ left$)$ rejective chain} if the following conditions hold for $1 \leq i \leq n$:
\begin{enumerate}
\item[(a)] $\mathcal{C}_i$ is a right (resp.\ left) rejective subcategory of $\mathcal{C}$;
\item[(b)] $\mathcal{C}_{i}$ is a cosemisimple subcategory of $\mathcal{C}_{i-1}$.
\end{enumerate}
\end{itemize}
\end{definition}

\begin{remark}
\begin{itemize}
\item[(1)] Rejective chains are total right rejective chains and total left rejective chains \cite[2.1(3)]{I2}.

\item[(2)] Our total right rejective chains are called right rejective chains in {\cite[Definition 2.6]{I3}}.
\end{itemize}
\end{remark}

\begin{example}
Let $A$ be the $\kk$-algebra given in Example \ref{eg}. 
Then 
\[
\proj A= \add A\supset \add (e_2 +e_3)A \supset \add e_3A \supset 0
\]
is a total right rejective chain of $\proj A$.
In fact, the conditions (a) and (b) in Definition \ref{rejch} (2) are satisfied by Proposition \ref{iy1} and $e_1J(A)e_1=0 =e_2J(A)e_2$ respectively.
\end{example}

Now, we are ready to prove the following main result.

\begin{theorem} \label{thm1}
Let $A$ be an artin algebra and
\begin{equation} \label{thm}
A=H_{0}>H_{1}> \cdots > H_i > \cdots >H_{n}=0
\end{equation}
a chain of idempotent ideals of $A$.
For $0 \leq i \leq n-1$, we write $H_i = Ae_iA$, where $e_i$ is an idempotent of $A$. 
Then the following conditions are equivalent:

\begin{itemize}
\item[{\rm (i)}] The chain \eqref{thm} is a right-strongly $($resp.\ left-strongly$)$ heredity chain.

\item[{\rm (ii)}] The following chain is a total right $($resp.\ left$)$ rejective chain of $\proj A$.
\[
\proj A=\add e_0 A \supset \add e_1 A \supset \cdots \supset \add e_n A =0.
\]
\end{itemize}

In particular, an artin algebra $A$ is strongly $($resp.\ right-strongly, left-strongly$)$ quasi-hereditary if and only if $\proj A$ has a rejective $($resp.\ total right rejective, total left rejective$)$ chain.
\end{theorem}

\begin{proof}
It follows from Proposition \ref{iy1} that $H_i \in \proj A$ if and only if $\add e_iA$ is a right rejective subcategory of $\proj A$.
Thus we have that $H_i$ satisfies the condition (a) in Definition \ref{sthered} (1) if and only if $\add e_iA$ satisfies the condition (a) in Definition \ref{rejch} (2).
From Lemma \ref{coss}, we have that $(H_i/H_{i+1})J(A/H_{i+1})(H_i/H_{i+1})=0$ holds if and only if $\add e_{i+1}A$ is a cosemisimple subcategory of $\add e_iA$. 
Thus we obtain that $H_i$ and $H_{i+1}$ satisfy the condition (b) in Definition \ref{sthered} (1) if and only if $\add e_iA$ and $\add e_{i+1}A$ satisfy the condition (b) in Definition \ref{rejch} (2).
\end{proof}

We apply Theorem \ref{thm1} to the following well-known result.

\begin{corollary} \label{Ri}
Let $A$ be an artin algebra. 

\begin{itemize}
\item[{\rm (1)}] {\rm (Iyama {\cite[Theorem 1.1]{I}})} For any $M \in \mod A$, there exists $N \in \mod A$ such that $\add N$ contains $M$ and has a total right rejective chain.

\item[{\rm (2)}] {\rm (Ringel {\cite[Theorem in \S 5]{R}})} There exists a right-strongly quasi-hereditary algebra $B$ and an idempotent $e$ of $B$ such that $A=e B e$.  
\end{itemize}
\end{corollary}

\begin{proof}
(1) For the reader's convenience, we recall the construction.
Let $M_0:= M$ and $M_{i+1}:= J(\mathrm{End}_{A}(M_i)) M_i$ inductively.
We take the smallest $n>0$ such that $M_n=0$, and let $N:= \bigoplus_{k=0}^{n-1} M_k$ and $\mathcal{C}_i:= \add (\bigoplus_{k=i}^{n-1} M_k)$.
Then 
\[
\mathcal{C}_{0}\supset\mathcal{C}_{1}\supset\cdots\supset \mathcal{C}_{n}=0
\]
is a total right rejective chain {\cite[Lemma 2,2]{I}}, {\cite[Theorem 3.4.1]{I3}}.

(2) Applying (1) to $M=A$, we obtain that $N \in \mod A$ such that $B=\End_A(N)$ is right-strongly quasi-hereditary by Theorem \ref{thm1}.
Let $e \in B$ be the idempotent corresponding to the direct summand $A$ of $N$.
Then $eBe=A$ holds as desired.
\end{proof}

Another application of Theorem \ref{thm1} is the following.

\begin{example}
Let $A$ be an artin algebra. Then there exists the smallest $n>0$ such that $J(A)^n=0$.
Let $\mathcal{C}_i:= \add \bigoplus_{k=1}^{n-i} A/J(A)^k$ and $B:= \mathrm{End}_{A}(\bigoplus_{k=1}^{n} A/J(A)^k)$. 
Then $B$ has finite global dimension \cite{A}.
 Moreover $B$ is quasi-hereditary \cite{DR3}.
On the other hand, 
\[
\proj A \simeq \mathcal{C}_0 \supset \mathcal{C}_1 \supset \cdots \supset \mathcal{C}_n =0.
\]
is a total left rejective chain {\cite[Example 2.7.1]{I3}}. 
Thus we obtain from Theorem \ref{thm1} that $B$ is a left-strongly quasi-hereditary algebra. 
This was independently shown in \cite{{R}, {C}}.
\end{example}

We end this subsection with characterizations of cosemisimple right (resp.\ left) rejective subcategories.
The first one is crucial in the proof of Corollary \ref{Ri} (1).

\begin{proposition} [Iyama {\cite[1.5.1]{I2}}] \label{crrs}
Let $\mathcal{C}$ be a Krull--Schmidt category and $\mathcal{C}'$ a subcategory of $\mathcal{C}$.
Then $\mathcal{C}'$ is a cosemisimple right $($resp.\ left$)$ rejective subcategory of $\mathcal{C}$ if and only if, for any $X \in \ind \mathcal{C} \setminus \ind \mathcal{C}'$,
there exists a morphism $\varphi : Y \to X$ $($resp.\ $\varphi: X \to Y)$ such that $Y \in \mathcal{C}'$ and $\mathcal{C}(-, Y) \xrightarrow{\varphi \circ -} \mathcal{J}_{\mathcal{C}}(-, X)$ $($resp.\ $\mathcal{C}(Y, -) \xrightarrow{-\circ \varphi} \mathcal{J}_{\mathcal{C}}(X, -))$ is an isomorphism on $\mathcal{C}$. 
\end{proposition}

\begin{proof}
We show the ``only if'' part.
For any $X \in \ind \mathcal{C} \setminus \ind \mathcal{C}'$, we take a morphism $\varphi : Y \to X$ such that $Y \in \mathcal{C}'$ and $\mathcal{C}(-, Y) \xrightarrow{\varphi \circ -} [\mathcal{C}'](-, X)$ is an isomorphism on $\mathcal{C}$.
This gives a desired morphism since cosemisimplicity of $\mathcal{C}'$ implies that $\mathcal{J}_\mathcal{C}(-, X) =[ \mathcal{C}'](-, X)$.

We show the ``if'' part.
It suffices to prove that $[\mathcal{C}'](-, X) = \mathcal{J}_{\mathcal{C}}(-, X)$ for any $X \in \ind \mathcal{C} \setminus \ind \mathcal{C}'$.
For any $X \in \ind \mathcal{C} \setminus \ind \mathcal{C}'$, we take a morphism $\varphi : Y \to X$ such that $Y \in \mathcal{C}'$ and $\mathcal{C}(-, Y) \xrightarrow{\varphi \circ -} \mathcal{J}_{\mathcal{C}}(-, X)$ is an isomorphism on $\mathcal{C}$. 
Then $\mathcal{J}_{\mathcal{C}}(-, X) \subseteq \mathrm{Im} (\varphi \circ -) \subseteq [\mathcal{C}'](-, X)$ holds.
Since $X \not\in \mathcal{C}'$, this clearly implies $\mathcal{J}_{\mathcal{C}}(-, X) = [\mathcal{C}'](-, X)$, and hence we have the assertion.
\end{proof}

The second one is a reformulation of Proposition \ref{crrs}.

\begin{proposition} [Iyama {\cite[Theorem 3.2(3)]{I2}}] \label{iy2} 
Let $A$ be a basic artin algebra and $e$ an idempotent of $A$. 
Then $\add e A$ is a cosemisimple right $($resp.\ left$)$ rejective subcategory of $\proj A$ if and only if $(1-e) J(A) \in \add e A$ as a right $A$-module $($resp.\ $J(A) (1-e) \in \add A e$ as a left $A$-module$)$. 
\end{proposition}

\begin{proof}
Applying Proposition \ref{crrs} to $\mathcal{C}:= \proj A$ and $\mathcal{C}':= \add e A$, we have that $\mathcal{C}'$ is a cosemisimple right rejective subcategory of $\mathcal{C}$ if and only if there exists a morphism $\varphi : Y \to (1-e) A$ with $Y \in \mathcal{C}'$ such that 
\[
Y \cong \mathcal{C}(A, Y) \xrightarrow{\varphi \circ -} \mathcal{J}_{\mathcal{C}}(A, (1-e) A )= (1-e) J(A)
\]
is an isomorphism.
This means that $(1-e) J(A) \in \mathcal{C}'$ holds.
\end{proof}

\subsection{Coreflective subcategories}

In this subsection, we study a weaker notion of right (resp.\ left) rejective subcategories called coreflective (resp.\ reflective) subcategories.
They appeared in the classical theory of localizations of abelian categories \cite{St}. 
Let us start with recalling their definitions.

\begin{definition}[Cf.\ Stenstr\"om \cite{St}] \label{refsub}
Let $\mathcal{C}$ be an additive category and $\mathcal{C}'$ a subcategory of $\mathcal{C}$.
We call $\mathcal{C}'$ a \emph{coreflective} (resp.\ \emph{reflective}) subcategory of $\mathcal{C}$ if the inclusion functor $\mathcal{C}'\hookrightarrow\mathcal{C}$
admits a right (resp.\ left) adjoint.
\end{definition}

Clearly right (resp.\ left) rejective subcategories are coreflective (resp.\ reflective).
The following proposition is an analogue of Proposition \ref{app}. 

\begin{proposition} \label{app2}
Let $\mathcal{C}$ be an additive category and $\mathcal{C}'$ a subcategory of $\mathcal{C}$.
Then the following conditions are equivalent:

\begin{itemize}
\item[{\rm (i)}] $\mathcal{C}'$ is a coreflective $($resp.\ reflective$)$ subcategory of $\mathcal{C}$.

\item[{\rm (ii)}] For any $X\in\mathcal{C}$, there exists a right $($resp.\ left$)$ $\mathcal{C}'$-approximation $f_X \in \mathcal{C}\left(Y,X\right)$ $($resp.\ $f^X \in \mathcal{C}\left(X,Y\right))$ of $X$ such that 
$
\mathcal{C}(-, Y) \xrightarrow{f_X \circ -} \mathcal{C}(-, X) \; \mathrm{(resp.\;} \mathcal{C}(Y,-) \xrightarrow{- \circ f^X} \mathcal{C}(X,-) )
$
is an isomorphism on $\mathcal{C}'$.
\end{itemize}
\end{proposition}

We omit the proof since it is similar to Proposition \ref{app}.

The following proposition is an analogue of Proposition \ref{iy1}.

\begin{proposition} [Iyama {\cite[Theorem 3.2 (1)]{I2}}] \label{iy3}
Let $A$ be an artin algebra and $e$ an idempotent of $A$. 
Then $\add e A$ is a coreflective $($resp.\ reflective$)$ subcategory of $\proj A$ if and only if $A e$ $($resp.\ $e A)$ is a projective right $($resp.\ left$)$ $eA e$-module.
\end{proposition}

\begin{proof}
Assume that $\add e A$ is a coreflective subcategory of $\proj A$.
Then there exists a right $(\add e A)$-approximation $a \in \Hom_A (P, A)$ of $A$ such that 
\[
\Hom_{A}(e A, P) \xrightarrow{a \circ -} \Hom_{A}(e A, A)
\]
is an isomorphism. 
Thus we have an isomorphism $A e \cong Pe \in \add (e A e)$ of right $(e A e)$-modules and we obtain $A e \in \proj (e A e)$.
 
Conversely, we assume that $A e$ is projective as right $(e A e)$-modules.
Then there exists $P \in \add e A$ as a right $A$-module such that $Pe \cong A e$ as right $(eA e)$-modules. 
This is induced by a morphism $a: P \to A$ since $\Hom_A(P, A) = \Hom_{eAe} (Pe, Ae)$ (see {\cite[Proposition 2.1 (a)]{ARS}}).
Since
\[
\Hom_{A}(e A, P) \xrightarrow{a \circ -} \Hom_{A}(e A, A)
\]
is an isomorphism, $\add e A$ is coreflective in $\proj A$. 
\end{proof}

Right (resp.\ left) rejective subcategories are coreflective (resp.\ reflective) subcategories, but the converse is not true as the following example shows.

\begin{example}
Let $A$ be the preprojective algebra of type $\mathbb{A}_3$.
It is defined by the quiver
\[
\xymatrix@C=25pt@R=15pt{1\ar@<2pt>[r]^{\alpha_1} & 2\ar@<2pt>[l]^{\beta_1} \ar@<2pt>[r]^{\alpha_2} &3\ar@<2pt>[l]^{\beta_2}
 }
 \]
with relations $\alpha_1 \beta_1, \beta_1 \alpha_1- \alpha_2 \beta_2, \beta_2 \alpha_2$.
Then $A e_3A$ is not projective as a right $A$-module, but $Ae_3$ is projective as a right $e_3 A e_3$-module.
Thus $\add e_3 A$ is not a right rejective subcategory of $\proj A$ by Proposition \ref{iy1}, but a coreflective subcategory of $\proj A$ by Proposition \ref{iy3}.
\end{example}

We introduce the following analogue of Definition \ref{rejch}.

\begin{definition} \label{refch}
Let $\mathcal{C}$ be an additive category. 
We call a chain of subcategories
\[
\mathcal{C}=\mathcal{C}_{0}\supset \mathcal{C}_1 \supset \cdots \supset \mathcal{C}_{n}=0
\]
a \emph{coreflective $($resp.\ reflective$)$ chain} if $\mathcal{C}_{i}$ is a cosemisimple coreflective $($resp.\ reflective$)$ subcategory of $\mathcal{C}_{i-1}$ for $1 \leq i \leq n$.
\end{definition}

Clearly right (resp.\ left) rejective chains are coreflective (resp.\ reflective) chains.
The converse is not true as the following example shows.

\begin{example}
Let $A$ be the preprojective algebra of type $\mathbb{A}_2$.  
It is defined by the quiver
\[
\xymatrix@C=25pt@R=15pt{1\ar@<2pt>[r]^{\alpha} & 2\ar@<2pt>[l]^{\beta}
 }
 \]
with relations $\beta \alpha, \alpha \beta$. 
Then
\[
\proj A =\add A \supset \add e_2A  \supset 0 
\]
is not a right rejective chain, but a coreflective chain of $\proj A$.
In fact, the conditions (a) and (b) in Definition \ref{refch} follow from Proposition \ref{iy3} and $e_1 J(A) e_1 =0$ respectively. 
However the condition (a) in Definition \ref{rejch} dose not hold by Proposition \ref{iy1}.
\end{example}

We are ready to state the following main result in this paper.

\begin{theorem} \label{thm2}
In Theorem \ref{thm1}, the conditions $(\mathrm{i})$ and $(\mathrm{ii})$ are equivalent to the following condition.

\begin{itemize}
\item[{\rm (iii)}] \eqref{thm} is a heredity chain of $A$ and the following chain is a coreflective $($resp.\ reflective$)$ chain of $\proj A$. 
\[
\proj A=\add e_0 A \supset \add e_1 A \supset \cdots \supset \add e_n A =0.
\]
\end{itemize}
\end{theorem}

To prove Theorem \ref{thm2}, we need the following lemma.

\begin{lemma} \label{lem3}
Let $A$ be an artin algebra and $I' \subset I$ idempotent ideals of $A$. 
Let $e$ and $e'$ be idempotents of $A$ such that $I=AeA$ and $I'=Ae'A$.
We assume that $I/I'$ is a projective right $($resp.\ left$)$ $(A/I')$-module and $\mathrm{Tor}^A_2(A/I, A/I')=0$ $($resp.\ $\mathrm{Tor}^A_2(A/I', A/I)=0)$.
If $Ae$ $($resp.\ $eA)$ is a projective right $($resp.\ left$)$ $(eAe)$-module, then $I$ is a projective right $($resp.\ left$)$ $A$-module.  
\end{lemma}

\begin{proof}
Let $0 \to K \to P \to I \to 0$ be a projective cover of the right $A$-module $I$.
Then $P \in \add eA$ as a right $A$-module and $K \subset P J(A)$ hold.
Applying the functor $(-)e : \mod A \to \mod eAe$, we have a short exact sequence $0 \to Ke \to Pe \to Ie \to 0$.
Since $Ie=Ae$ and $Pe$ are projective $(eAe)$-modules and $Ke \subset Pe J(eAe)$, we have $Ke=0$.

On the other hand, applying the functor $- \otimes_A (A/I')$ to the short exact sequence $0 \to K \to P \to I$, we have an exact sequence 
\[
\mathrm{Tor}_1^A(I, A/I') \to K/KI' \to P/PI' \to I/I' \to 0,
\]
where $\mathrm{Tor}_1^A(I, A/I')= \mathrm{Tor}_2^A(A/I, A/I')=0$ holds by our assumption.
Since $I/I'$ is a projective right $(A/I')$-module, the sequence splits, and hence $K/KI'$ is a direct summand of $P/PI'$.
On the other hand, $K \subset PJ(A)$ implies that $K/KI' \subset (P/PI') J(A)$. 
Thus $K/KI'=0$ holds.
Consequently, $K=KI' \subset KI=0$ holds as desired.
\end{proof}

We are ready to prove Theorem \ref{thm2}.

\begin{proof} [Proof of Theorem \ref{thm2}]
Since (ii) $\Rightarrow$ (iii) clearly holds, it suffices for us to prove that (iii) $\Rightarrow$ (i).
We show this claim by induction on $n$. 
If $n=1$, then the assertion holds since $H_0=A$ is projective as a right $A$-module.

For $n \geq 2$ we proceed by induction. 
Let $e_i$ denote the idempotent $e_i +H_{n-1}$ of $A/H_{n-1}$ for $0 \leq i \leq n-2$.  
Firstly, we claim that 
\begin{equation*}
A/H_{n-1}> \cdots >(A/H_{n-1})e_i (A/H_{n-1})  >\cdots >H_{n-1}/H_{n-1}=0
\end{equation*}
is a heredity chain of $A/H_{n-1}$ such that $\add e_i (A/H_{n-1})$ is a coreflective subcategory of $\proj (A/H_{n-1})$ for $0 \leq i \leq n-2$.
Since $(A/H_{n-1}) e_i (A/H_{n-1}) = H_i/H_{n-1}$ for $0 \leq i \leq n-2$, the above chain is a heredity chain of $A/H_{n-1}$.
Since $Ae_i \in \proj (e_i A e_i)$, we have that $Ae_i \otimes_{e_iAe_i} e_i(A/H_{n-1})e_i=(A/H_{n-1})e_i$ is projective as a right $(e_i (A/H_{n-1}) e_i)$-module.
Therefore it follows from Proposition \ref{iy3} that $\add e_i (A/H_{n-1})$ is a coreflective subcategory of $\proj (A/H_{n-1})$ for $0 \leq i \leq n-2$.

Now, we deduce from the induction hypothesis that $H_i/H_{n-1}$ is a projective module as a right $(A/H_{n-1})$-module for $0 \leq i \leq n-2$.
For any $0 \leq i \leq n-1$, we obtain from the hypothesis (iii) that $\add e_i A$ is a coreflective subcategory of $\proj A$, and hence $A e_i$ is a projective right $(e_i A e_i)$-module by Proposition \ref{iy3}. 
Thus we have idempotent ideals $H_{n-1}, H_i$ such that $H_{n-1}$ is a heredity ideal of $A$, $H_i/H_{n-1}$ is projective as a right $(A/H_{n-1})$-module and $A e_i$ is a projective right $(e_iAe_i)$-module for $0 \leq i \leq n-1$. 
Moreover $\mathrm{Tor}^A_2(A/H_i, A/H_{n-1})=\mathrm{Tor}^{A/H_{n-1}}_2(A/H_i, A/H_{n-1})=0$ holds since $H_{n-1}$ is a heredity ideal of $A$.
Therefore we deduce from Lemma \ref{lem3} that $H_i$ is a projective right $A$-modules.
\end{proof}

\section{Algebras of global dimension at most two}

\subsection{Algebras of global dimension at most two are right-strongly quasi-hereditary}

The aim of this subsection is to prove the following result.

\begin{theorem} \label{gl}
Let $A$ be an artin algebra such that $\gl A \leq2$.
Then the following statements hold.

\begin{itemize}
\item[{\rm (1)}] $A$ is a right-strongly quasi-hereditary algebra.

\item[{\rm (2)}] {\rm (Iyama {\cite[Theorem 3.6]{I2}}).}
The category $\proj A$ has a total right rejective chain
\[
\proj A=\add e_0 A \supset \add e_1 A \supset \cdots \supset \add e_n A =0.
\]
\end{itemize}
\end{theorem}

Note that for an artin algebra $A$ of global dimension at most two, we can similarly construct a total left rejective chain 
\[
\proj A=\add \epsilon_0 A \supset \add \epsilon_1 A \supset \cdots \supset \add \epsilon_n A =0.
\]  
Hence such an algebra is left-strongly quasi-hereditary.
However it is not necessarily strongly quasi-hereditary.

We need the following preparation.

\begin{lemma} [Iyama {\cite[Lemma 3.6.1]{I2}}] \label{sim}
Let $A$ be an artin algebra with $\gl A =m$ where $2 \leq m < \infty$.
Then there exists simple right $A$-modules $S$ and $S'$ such that the projective dimensions of $S$ and $S'$ are $m-1$ and $m$ respectively.
\end{lemma}

\begin{proof}
Existence of $S'$ is clear since $\gl A$ is supremum of the projective dimensions of simple $A$-modules.
Let $0 \to X \to P \to S' \to 0$ be an exact sequence with a projective $A$-module $P$.
Then the projective dimension of $X$ is precisely $m-1$.
We assume that $X$ is not simple.
Then there exists a proper simple submodule $L$ of $X$.
Consider the short exact sequence
\begin{equation*}
0 \to L \to X \to X/L \to 0.
\end{equation*}
Since the projective dimension of $X$ is $m-1$ and $\gl A=m$, the projective dimension of $L$ is at most $m-1$.
We assume that the projective dimension of $L$ is strictly less than $m-1$.
Then the projective dimension of $X/L$ is precisely $m-1$.
Therefore we obtain the assertion by replacing $X$ by $X/L$ and repeating this argument. 
\end{proof}

We are ready to prove the main theorem in this subsection.

\begin{proof}[Proof of Theorem \ref{gl}]
(2) We show by induction on the number of simple modules.
We may assume that $A$ is basic. 
Let $n$ be the number of simple $A$-modules.  

Assume that $n=1$. 
Since $A$ is simple, the assertion holds.

For $n \geq 2$ we proceed by induction. 
If $A$ is semisimple, then the assertion is obvious.
Thus we assume that $A$ is non-semisimple.
It follows from Lemma \ref{sim} that there exists a simple $A$-module $S$ such that the projective dimension of $S$ is precisely one since $\gl A=1$ or $\gl A=2$.
Let $f$ be a primitive idempotent of $A$ such that $S=f(A/J(A))$. 
Let $e:=1-f$ and $A':=eAe$.

(i) We claim that $\add eA$ is a cosemisimple right rejective subcategory of $\proj A$ and $\gl A' \leq \gl A \leq 2$.
There exists a short exact sequence
\begin{equation*}
0 \to f J(A) \xrightarrow{\varphi} f A \to S \to 0.
\end{equation*}
Since the projective dimension of $S$ is one, we have $f J(A) \in \proj A$.
Since $fA$ is not an indecomposable direct summand of $f J(A)$, we have $f J(A)  \in \add eA$ as a right $A$-module.
It follows from Proposition \ref{iy2} that $\add eA$ is a cosemisimple right rejective subcategory of $\proj A$.
Thus $A/AeA$ is simple.
Since $\add e A$ is a right rejective subcategory of $\proj A$, it follows from Proposition \ref{iy1} that $\gl A' \leq \gl A \leq 2$.

(ii) We claim that any monomorphism in $\add eA$ is monic in $\proj A$.
Let $a: P_1 \to P_0$ be a monomorphism in $\add eA$.
Then we have an exact sequence 
\begin{equation*}
0 \to \ker a \to P_1 \xrightarrow{a} P_0 \to \cok a \to 0
\end{equation*} 
in $\mod A$.
Since $\gl A \leq 2$, we obtain that $P_2 := \ker a \in \proj A$.
Since $a$ is a monomorphism in $\add e A$, we have $P_2 e= \Hom_{A}(eA, P_2) =0$.
This implies that $P_2$ is a module over a simple algebra $A/A e A$.
Thus we obtain that $P_2$ is isomorphic to $S^l$ for some $l \geq 0$.
If $l>0$, then $S$ is projective as a right $A$-module.
This is a contradiction since the projective dimension of $S$ is one.
Therefore we have $l=0$ and $P_2=0$.
Thus $a$ is a monomorphism of $A$-modules, and hence the assertion follows.

(iii) We claim that any right rejective subcategory $\mathcal{C}$ of $\add eA$ is also right rejective in $\proj A$.
In fact, $A=eA \oplus fA$ and $\add eA$ has a right $\mathcal{C}$-approximation which is monic in $\add eA$ and hence it is also monic in $\proj A$ by (ii).
Similarly, composing a right $\mathcal{C}$-approximation of $fJ(A) \in \add eA$ and $\varphi: fJ(A) \hookrightarrow fA$, we have a right $\mathcal{C}$-approximation of $fA$ which is monic in $\add eA$, and hence in $\proj A$ by (ii).
 
(iv) We complete the proof by induction on the number of simple $A$-modules.
By induction hypothesis, $\proj A' \simeq \add eA$ has a total right rejective chain. 
\[
\proj A' \simeq \add eA \supset \mathcal{C}_1 \supset \cdots \supset \mathcal{C}_{n-1} \supset \mathcal{C}_{n}=0.
\]
Composing it with $\proj A \supset \add eA$, and apply (iii), we have a total right rejective chain of $\proj A$.

(1) The assertion follows from (2) and Theorem \ref{thm1}.
\end{proof}

If $A$ is a strongly quasi-hereditary algebra, then the global dimension of $A$ is at most two {\cite{R}}. 
The converse is not true as the following example shows.

\begin{example} \label{cex}
Let $Q$ be the quiver $1 \gets 2 \to 3$ whose underlying graph is the Dynkin graph $\mathbb{A}_3$ and $A$ the Auslander algebra of $\kk Q$.
Then $A$ is defined by the quiver 
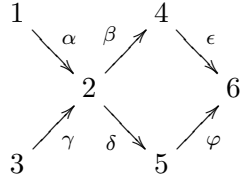
\begin{eqnarray*} 
\xymatrix@=15pt{1 \ar[rd]^{\alpha} && 4 \ar[rd]^{\epsilon} \\
& 2 \ar[ru]^{\beta} \ar[rd]_{\delta} && 6 \\
3 \ar[ru]_{\gamma} && 5 \ar[ru]_{\varphi}
 }
\end{eqnarray*}
with relations $\alpha \beta$, $\gamma \delta$ and $\beta \epsilon -\delta \varphi$.
The global dimension of $A$ is two. However we can not construct a strongly heredity chain of $A$.
This  example can be explained by Theorem \ref{Aus}.
\end{example}  

We end this subsection with describing a certain class of artin algebras which is called Ringel self-dual.
We recall the following result.

\begin{definition-theorem} [Ringel {\cite[Theorem 5]{R2}}] \label{rin}
Let $A$ be a quasi-hereditary algebra and $I:=\{ 1 < \cdots < n \}$. 
Then there exist the indecomposable $A$-modules $T(1), T(2), \ldots, T(n)$ such that
$T(i)$ is Ext-injective in $\mathcal{F}(\Delta)$ and the standard module $\Delta(i)$ is embedded to $T(i)$, with $T(i)/\Delta(i) \in \mathcal{F}(\Delta(j) \;| \; j<i)$.
Let $T:= \bigoplus_{i=1}^n T(i)$ and $R(A):= \End_A(T)^{\op}$.
Then $R(A)$ is a quasi-hereditary algebra with respect to the opposite order of $\leq$. 
We call $R(A)$ a \emph{Ringel dual} of $A$.
\end{definition-theorem}

Let $(A, \leq_A)$ and $(B, \leq_B)$ be quasi-hereditary algebras with simple $A$-modules $\{S_A(i) \; | \; i \in I\}$ and simple $B$-modules $\{S_B(i') \; | \; i' \in I'\}$.
We say that {\it $(A, \leq_A)$ is isomorphic to $(B, \leq_B)$ as a quasi-hereditary algebra} if there exists an algebra isomorphism $f :A \xrightarrow{\sim} B$ such that the induced map $\varphi: I \to I'$ is a poset isomorphism.

Let $(A, \leq)$ be quasi-hereditary algebra and $\leq^{\op}$ the opposite order of $\leq$.
We say that $A$ is {\it Ringel self-dual} if $(A, \leq_A)$ is isomorphic to $(R(A), \leq^{\op})$ as a quasi-hereditary algebra. 

\begin{corollary} \label{rdual}
Let $A$ be a Ringel self-dual algebra.
Then the following conditions are equivalent:

\begin{itemize}
\item[{\rm (i)}] $A$ has global dimension at most two.

\item[{\rm (ii)}] $A$ is strongly quasi-hereditary.

\item[{\rm (iii)}] $A$ is right strongly quasi-hereditary. 
\end{itemize}
\end{corollary}

\begin{proof}
(ii) $\Rightarrow$ (i): This is shown in {\cite{R}}. 

(i) $\Rightarrow$ (iii): This follows from Theorem \ref{gl} immediately. 

(iii) $\Rightarrow$ (ii): Let $A$ be a right strongly quasi-hereditary algebra. 
Since the Ringel dual of a right-strongly quasi-hereditary algebra is left-strongly quasi-hereditary with the opposite order {\cite[Proposition A.2]{R}}, $A$ is strongly quasi-hereditary. 
\end{proof}

\subsection{Strongly quasi-hereditary Auslander algebras}
Since the global dimensions of Auslander algebras are at most two, it follows from Theorem \ref{gl} that each Auslander algebra is right-strongly quasi-hereditary.
However it is not necessarily true that each Auslander algebra is strongly quasi-hereditary. 
The aim of this subsection is to provide the following characterization of Auslander algebras which are strongly quasi-hereditary. 
Recall that an artin algebra $A$ is a Nakayama algebra if and only if every indecomposable $A$-module is uniserial (see for example {\cite[\S 4.2]{ARS}}).

\begin{theorem} \label{Aus}
Let $A$ be a representation-finite artin algebra and $B$ the Auslander algebra of $A$.
Then the following conditions are equivalent.

\begin{itemize}
\item[{\rm (i)}] $B$ is strongly quasi-hereditary.

\item[{\rm (ii)}] $\proj B$ has a rejective chain.

\item[{\rm (iii)}] $A$ is a Nakayama algebra.
\end{itemize}
\end{theorem}

To prove Theorem \ref{Aus}, we need the following observation.

\begin{lemma} \label{lNak}
Let $A$ be an artin algebra and $B$ a factor algebra of $A$ such that $\mod B$ is a cosemisimple subcategory of $\mod A$.
Then the following statements hold.

\begin{itemize}
\item[{\rm (1)}] Let $X$ be an indecomposable $A$-module which does not belong to $\mod B$.
Then $X$ is a projective-injective $A$-module such that $XJ(A)$ is an indecomposable $B$-module.

\item[{\rm (2)}] $B$ is a Nakayama algebra if and only if $A$ is a Nakayama algebra.
\end{itemize}
\end{lemma}

\begin{proof}
(1) By Proposition \ref{crrs}, there exists a morphism $\varphi : Y\to X$ of $A$-modules such that $Y \in \mod B$ and $\Hom_A(-,Y) \xrightarrow{\varphi \circ -} \mathcal{J}_{\mod A}(-,X)$ is an isomorphism on $\mod A$.
Then $\varphi$ is a minimal right almost split morphism of $X$ in $\mod A$.
If $X$ is not a projective $A$-module, then $\varphi$ is surjective and hence $X\in\mod B$, a contradiction.
Therefore $X$ is a projective $A$-module, and $\varphi$ is an inclusion map $XJ(A) \to X$.
Thus $XJ(A)=Y$ is a $B$-module.
The dual argument shows that $X$ is an injective $A$-module, and hence $XJ(A)$ is indecomposable.

(2) Since the ``if'' part is obvious, we prove the ``only if'' part.
Let $M$ be an indecomposable $A$-module which is either projective or injective.
We show that $M$ is a uniserial $A$-module.
If $M$ is a $B$-module, then this is clear. 
Assume that $M$ is not a $B$-module.
By (1), $M$ is a projective-injective $A$-module such that $M J(A)$ is an indecomposable $B$-module.
Since $B$ is a Nakayama algebra, $M J(A)$ is uniserial.
Hence $M$ is also uniserial.
\end{proof}

We are ready to show the main theorem in this subsection.

\begin{proof}[Proof of Theorem \ref{Aus}]
It suffices to show from Theorem \ref{thm1} that (ii) is equivalent to (iii).

(ii) $\Rightarrow$ (iii): We show by induction on the length $l(A)$ of $A$ as a right $A$-module.
If $l(A)=1$, then this is clear.
For $l(A) \geq 2$ we proceed by induction.
Since $B$ is a strongly quasi-hereditary algebra, it follows from Theorem \ref{thm1} that $\proj B \simeq \mod A$ has a rejective chain
\[
\mod A \supset \mathcal{C}_1 \supset \cdots \supset \mathcal{C}_n \supset 0.
\]
Since $\mathcal{C}_1$ is a rejective subcategory of $\mod A$, there exists a two-sided ideal $I$ of $A$ such that $\mathcal{C}_1=\mod (A/I)$ by Proposition \ref{bij}.
It follows from the induction hypothesis that $A/I$ is a Nakayama algebra.
Therefore we obtain from Lemma \ref{lNak} (2) that $A$ is also a Nakayama algebra.

(iii) $\Rightarrow$ (ii): We show by induction on $l(A)$.
If $l(A)=1$, then the assertion holds.
For $l(A) \geq 2$, we prove that $\mod A$ has a rejective chain by induction.
Since $A$ is a Nakayama algebra, there exists an indecomposable projective-injective $A$-module $P$.
Let $M$ be a direct sum of all indecomposable $A$-modules which are not isomorphic to $P$ and $\mathcal{C}_1 := \add M$.
Then $\mathcal{C}_1$ is closed under factor modules and submodules.
It follows from Proposition \ref{bij} that there exists a two-sided ideal $I$ of $A$ such that $\mathcal{C}_1 = \mod (A/I)$.
On the other hand, we have $\ind (\mod A) \setminus \ind (\mathcal{C}_1)= \{P \}$.
Since the inclusion map $\varphi : PJ(A) \to P$ gives an isomorphism $\Hom_A(-, PJ(A)) \xrightarrow{\varphi \circ -} \Hom_A(-, P)$ on $\mod A$, we obtain from Proposition \ref{crrs} that $\mod A/I$ is a cosemisimple right rejective subcategory of $\mod A$.
Since $l(A)> l(A/I)$, we obtain from the induction hypothesis that there exists a rejective chain $\mod (A/I) \supset \cdots \supset \mathcal{C}_i \supset \cdots \supset 0$ of $\mod (A/I)$.
Composing with $\mod A \supset \mod (A/I)$, we have a rejective chain of $\mod A$. 
\[
\mod A \supset \mathcal{C}_1 = \mod (A/I) \supset \cdots \supset \mathcal{C}_n \supset 0.
\]
The proof is complete.
\end{proof}

\subsection*{Acknowledgement}
The author is greatly indebted to Osamu Iyama for his insightful comments and permission to include detailed proof of results in \cite{I2} in this paper.
The author would like to thank Kentaro Wada, Ryoichi Kase and Yasuaki Ogawa for their valuable comments.
The author would also like to thank Hyohe Miyachi and Yoshiyuki Kimura for their assistance.


\begin{thebibliography}{10}

\bibitem{ASS}
I.~Assem, D.~Simson, A.~Skowro$\acute{\textnormal n}$ski, \emph{Elements of the Representation Theory of Associative Algebras. Vol. 1}, London Mathematical Society Student Texts {\bf 65}, Cambridge university press (2006).


\bibitem{A}
M.~Auslander,
\newblock {\em Representation dimension of Artin algebras},
\newblock Queen Mary College, 1971.

\bibitem{ARS}
M.~Auslander, I.~Reiten, S.~O.~Smal\o, \emph{Representation theory of Artin algebras}, Cambridge Studies in Advanced Mathematics, {\bf 36}. Cambridge University Press, Cambridge, 1995.

\bibitem{AS}
M.~Auslander and S.~O. Smal\o.
\newblock {\em Preprojective modules over {A}rtin algebras}.
\newblock  J. Algebra, {\bf 66} (1980), no. 1, 61-122.

\bibitem{CPS}
E.~Cline, B.~Parshall, and L.~Scott,
\newblock {\em Finite-dimensional algebras and highest weight categories},
\newblock J. Reine Angew. Math. {\bf 391} (1988), 85--99.

\bibitem{C}
T.~Conde,
\newblock {\em The quasihereditary structure of the {A}uslander--{D}lab--{R}ingel algebra,}
\newblock J. Algebra {\bf 460} (2016), 181--202.


\bibitem{C2}
T.~{Conde},
\newblock {\em $\Delta$-filtrations and projective resolutions for the Auslander-Dlab-Ringel algebra,}
\newblock Algebr. Represent. Theory {\bf 21} (2018), no. 3, 605--625.

\bibitem{DR3}
V.~Dlab and C.~M. Ringel,
\newblock {\em Every semiprimary ring is the endomorphism ring of a projective
  module over a quasihereditary ring,}
\newblock Proc. Amer. Math. Soc. {\bf 107} (1989), no. 1, 1--5.

\bibitem{DR}
V.~Dlab and C.~M. Ringel,
\newblock {\em Quasi-hereditary algebras},
\newblock Illinois J. Math. {\bf 33} (1989), no. 2, 280--291.

\bibitem{DR6}
V.~Dlab and C.~M. Ringel,
\newblock {\em The module theoretical approach to quasi-hereditary algebras},
\newblock Representations of algebras and related topics ({K}yoto,
  1990), 200--224, London Math. Soc. Lecture Note Ser., 168, Cambridge Univ. Press, Cambridge, 1992.

\bibitem{D}
S.~Donkin.
\newblock {\em The {$q$}-{S}chur algebra}, {\em London
  Mathematical Society Lecture Note Series, 253}.
\newblock Cambridge University Press, Cambridge, 1998.

\bibitem{DoRe}
S.~Donkin and I.~Reiten.
\newblock {\em On {S}chur algebras and related algebras. {V}. {S}ome
  quasi-hereditary algebras of finite type.}
\newblock  J. Pure Appl. Algebra, {\bf 97} (1994), no. 2, 117--134.

\bibitem{E}
\"O. Eir{\'\i}ksson.
\newblock {\em From submodule categories to the stable {A}uslander algebra.}
\newblock J. Algebra, {\bf 486} (2017) 98--118.

\bibitem{Er}
K.~Erdmann.
\newblock {\em Schur algebras of finite type}.
\newblock Quart. J. Math. Oxford Ser. (2), {\bf 44} (1993), no. 173, 17--41.

\bibitem{GLS}
C.~{Geiss}, B.~{Leclerc}, and J.~{Schr{\"o}er}.
\newblock {\em Cluster algebra structures and semicanonical bases for unipotent
  groups}.
\newblock arXiv:math/0703039v4.

\bibitem{I}
O.~Iyama,
\newblock {\em Finiteness of representation dimension},
\newblock Proc. Amer. Math. Soc. {\bf 131} (2003), no. 4, 1011--1014.

\bibitem{I2}
O.~{Iyama},
\newblock {\em Rejective subcategories of artin algebras and orders},
\newblock arXiv:math/0311281.

\bibitem{I3}
O.~Iyama.
\newblock {\em Representation dimension and {S}olomon zeta function}.
\newblock Representations of finite dimensional algebras and related
  topics in {L}ie theory and geometry, 45--64. Fields Inst., Commun., 40, Amer. Math. Soc., Providence, RI, 2004.

\bibitem{IR}
O.~Iyama and I.~Reiten.
\newblock {\em 2-{A}uslander algebras associated with reduced words in {C}oxeter
  groups}.
\newblock Int. Math. Res. Not. IMRN 2011, no. 8, 1782--1803.

\bibitem{PS}
B.~Parshall and L.~Scott.
\newblock {\em Derived categories, quasi-hereditary algebras, and algebraic groups}.
\newblock Carlton University Mathematical notes 3 (1988), 1--104.

\bibitem{R2}
C.~M. Ringel.
\newblock {\em The category of modules with good filtrations over a quasi-hereditary
  algebra has almost split sequences}.
\newblock Math. Z., {\bf 208} (1991), no. 2, 209--223.

\bibitem{R}
C.~M. Ringel.
\newblock {\em Iyama's finiteness theorem via strongly quasi-hereditary algebras}.
\newblock J. Pure Appl. Algebra, {\bf 214} (2010), no. 9, 1687--1692.

\bibitem{S}
L.~Scott.
\newblock {\em Simulating algebraic geometry with algebra. {I}. {T}he algebraic
  theory of derived categories}.
\newblock The {A}rcata {C}onference on {R}epresentations of {F}inite
  {G}roups ({A}rcata, {C}alif., 1986), 271--281, Proc. Sympos. Pure
  Math., 47, Part 2, Amer. Math. Soc., Providence, RI, 1987.

\bibitem{St}
B.~Stenstr\"om.
\newblock {\em Rings of quotients}.
\newblock Die Grundlehren der Mathematischen Wissenschaften, Band 217, An introduction to methods of ring theory.
\newblock Springer-Verlag, New York-Heidelberg, 1975.

\bibitem{UY}
M.~Uematsu and K.~Yamagata.
\newblock {\em On serial quasi-hereditary rings}.
\newblock Hokkaido Math. J., {\bf 19} (1990), no. 1, 165--174.

\end{thebibliography}
\end{document}